\documentclass[12pt]{amsart}
\usepackage{epsfig}
\usepackage{pdfsync}
\newtheorem{thm}{Theorem}[section]
\newtheorem{cor}[thm]{Corollary}
\newtheorem{lem}[thm]{Lemma}

\theoremstyle{definition}

\theoremstyle{remark}
\newtheorem{rem}[thm]{Remark}
\numberwithin{equation}{section}

\def\L2{L^{2}}

\def\N{\Bbb{N}}
\def\R{\Bbb{R}}

\def\e{\varepsilon}

\def\m1{^{-1}}

\def\S{\mathcal{S}}

\begin{document}

\title[]{Negative definite functions on groups\\ with polynomial growth}%
\author{Fabio Cipriani.}
\address{Dipartimento di Matematica, Politecnico Milano Piazza Leonardo da Vinci 33 20132 Milano Italy}
\email{fabio.cipriani@polimi.it}
\thanks{This work was supported by Italy I.N.D.A.M. – France C.N.R.S. G.D.R.E.-G.R.E.F.I. Geometrie Noncommutative and by Italy M.I.U.R.-P.R.I.N. project N. 2012TC7588-003.}
\author{Jean-Luc Sauvageot}
\address{Institut de Math\'ematiques, CNRS Universit\'e Denis Diderot F-75205 Paris Cedex 13, France}
\email{jlsauva@math.jussieu.fr}
\date{December 8, 2015}
\begin{abstract}
The aim of this work is to show that on a locally compact, second countable, compactly generated group $G$ with polynomial growth and homogeneous dimension $d_h$, there exist a continuous, proper, negative definite function $\ell$ with polynomial growth dimension $d_\ell$ arbitrary close to $d_h$.
\end{abstract}
\maketitle
\section{Introduction and statement of the results}
The length function $\ell$ associated to a finite set of generators of a countable, discrete, finitely generated group $\Gamma$, may be used to reveal interesting aspects of the group itself. This is the spirit of metric group theory (see [12]). In a famous result of U. Haagerup [11], for example, the length functions on the free groups $\mathbb{F}_n$ with $n\ge 2$ generators, has been used to prove that, even if these groups are not amenable, they still have an approximation property: there exists a sequence $\varphi_n\in c_0(\mathbb{F}_n)$ of normalized, positive definite functions, converging pointwise to the constant function$1$. Correspondingly, the trivial representation is weakly contained in a $C_0$-representation (one whose coefficients vanish at infinity) (see [10]). These properties of $\mathbb{F}_n$ only depend to the fact that the length function is {\it negative definite and proper}. The class of groups where a length function has the latter properties has been characterized in several independent ways and form the object of intensive investigations (see [2]). On another, somewhat opposite, side, a deeply studied class of groups are those possessing the {Kazhdan property (T)}, i.e. those where {\it every negative definite function is bounded} (see for example [1]). Finer asymptotic properties of length functions are connected to other fundamental properties of groups. For example, if $\ell$ has {\it polynomial growth} (see below for the definition) then $\Gamma$ is {\it amenable} and its von Neumann algebra $vN(\Gamma)$ is {\it hyperfinite} (weak closure on an increasing sequence of matrix algebras).
\par\noindent
However, length functions are not necessarily negative definite and one may wonder how far is a given length function from being negative definite.
\par\noindent
Reversing the point of view, in the present work we prove that on finitely generated, discrete groups $\Gamma$ or, more generally, on compactly generated, locally compact, second countable groups $G$, with polynomial growth, there always exists a continuous, negative definite function with a (polynomial) growth arbitrarily close to the one of $\Gamma$ or $G$.
\par\noindent
The consideration of negative definite functions and their growth properties are also meaningful in Noncommutative Potential Theory (see [3], [6]) and Quantum Probability (see [4]).
\par\noindent
As an application of the main results, we provide a natural construction of a noncommutative Dirichlet form on the group von Neumann algebra $\lambda(G)''$ with an upper spectral growth dimension $d_S$ bounded above by the polynomial growth homogeneous dimension $d_h$ of $G$.
\vskip0.2truecm
Let $G$ be a locally compact, second countable, compactly generated group with identity $e\in G$ and let $\mu$ be one of its
left Haar measures.
\vskip0.2truecm\noindent
The main results of the work are the following.

\begin{thm}\label{thm1}
Suppose that $G$ has polynomial growth and homogeneous dimension $d_h$. Then, for all $d>d_h$, there exists a continuous, proper, negative type function $\ell$ with polynomial growth such that
\begin{equation}\label{cresc}
\mu\{s\in G\,|\,\ell(s)\leq x\} =O(x^d)\,, \quad x\to +\infty.
\end{equation}
\end{thm}
\begin{thm}\label{thm2}
Suppose that $G$ has polynomial growth and homogeneous dimension $d_h$. Then, there exists a continuous, proper, negative type function $\ell$ with polynomial growth such that
\begin{equation}\label{crescpiu}
\forall d\,>d_h\, ,\quad \mu\{s\in G\,|\,\ell(s)\leq x\} =O(x^d)\quad \quad x\to +\infty.
\end{equation}
\end{thm}
As a straightforward application of these results to Noncommutative Potential Theory we have the following corollary. Further investigations will be discussed in [7].
\begin{cor}
Let $\Gamma$ be a discrete, finitely generated group with polynomial growth and homogeneous dimension $d_h$. Let $\lambda(\Gamma)''$ be the von Neumann algebra generated by the left regular representation and let $\tau$ be its trace.
\par\noindent Then there exists on $L^2(\lambda(\Gamma)'',\tau)$ a noncommutative Dirichlet form $(\mathcal{E},\mathcal{F})$ with discrete spectrum and upper spectral dimension $d_S \le d_h$.
\end{cor}\noindent
Here the {\it upper spectral dimension} $d_S$ of the Dirichlet form is defined as
\[
d_S:=\inf\{d>0:\limsup_{x\to +\infty}\sharp\{\lambda\in {\rm sp\,}(L): \lambda\le x\}/x^d<+\infty\}
\]
where ${\rm sp}(L)$ denotes the spectrum of the self-adjoint, nonnegative operator $(L,{\rm dom}(L))$ associated to $(\mathcal{E},\mathcal{F})$
\[
\mathcal{E}[a]=\|\sqrt{L}a\|^2_2\qquad a\in\mathcal{F}={\rm dom}(L)\, .
\]

\section{Groups with polynomial growth}
Let $K\subset G$ be a compact, symmetric, generator set with non empty interior
\[
G=\bigcup_{n=1}^\infty K^n\, ,\quad K^{-1}=K\,,\quad K^\circ\neq \emptyset\, ,
\]
where
\[
K^{-1}:=\{s^{-1}\in G:s\in K\}\, ,\qquad K^n:=\{s_1\dots s_n\in G:s_k\in K\, ,k=1,\dots ,n\}\, .
\]
We may always assume that $K$ is a neighborhood of the identity $e\in K^\circ\subset G$.
\par\noindent
Let us recall that $G$ has {\it polynomial growth} if
\begin{equation}
\exists\, c,d>0\quad {\rm such,\,\ that} \quad\mu(K^n)\leq c\,(n+1)^d\,,\qquad n\ge 1
\end{equation}
or, equivalently, if
\begin{equation}
\label{d}\frac{1}{\log(n+1)}\log(\mu(K^n))\leq d+\frac{1}{\log(n+1)}\log(c)\,,\qquad n\ge 1
\end{equation}
and in that case its {\it homogeneous dimension} is defined as
\begin{equation}
d_h=\limsup_{n\to \infty} \frac{1}{\log n}\log(\mu(K^n))\, .
\end{equation}
The class of locally compact, second countable, compactly generated groups with polynomial growth includes nilpotent, connected, real Lie groups and finitely generated, nilpotent, countable discrete groups (see [8] and [12] Chapter VII $\S$ 26). For example, the homogeneous dimension of the discrete Heisenberg group is $4$ (see [12] Chapter VII $\S$ 21).
\vskip0.2truecm\noindent
Recall also the M. Gromov's characterization [9] of finitely generated, countable discrete groups with polynomial growth as those which have a nilpotent subgroup of finite index (see also [12] Chapter VII $\S$ 29 and [13] for a generalization to locally compact groups).
\vskip0.2truecm\noindent
\begin{rem}
Let us observe that since $K$ is a neighborhood of the identity $e\in G$, we have $K^{n-1}\subset K^{n-1} K^\circ\subset (K^n)^\circ$ so that
\[
\bigcup_{k=0}^\infty (K^n)^\circ=G
\]
and we have an open cover of $G$. As a consequence
\vskip0.1truecm
\centerline{\it every compact set in $G$ is contained in $K^n$ for some $n\ge 1$.}
\end{rem}
\section{Proof of the Theorems}
In the following, we will consider the sequence $\{\alpha_n\}_{n=1}^\infty\subset [0,+\infty)$ defined by
\[
\alpha_1 :=0\, ,\qquad \frac{\mu(K^n)}{\mu(K^{n-1})}=:1+\alpha_n\, ,\qquad n\geq 2\, .
\]
Equation (\ref{d}) can thus be written as
\begin{equation}\label{d1}
\frac{1}{\log(n+1)} \sum_{k=1}^n \log(1+\alpha_k)\leq d+\frac{1}{\log(n+1)}\log(c)\leq d'\,,\qquad n\ge 1\, ,
\end{equation}
with $d':=d+\log_2 c$.

\subsection{Some preliminary estimates}
The proof of the theorems relies on the following lemmas.
\begin{lem}
For any $\beta\in (0,1)$ the set
\[
E_\beta :=\big\{n\in \N^*\,:\,\alpha_n> n^{-\beta}\big\}
\]
has vanishing density.
\end{lem}
\begin{proof}
Let us fix $K>0$ such that $\log(1+n^{-\beta})\geq K n^{-\beta}$ for all $n\ge 1$. By equation (\ref{d1}), we have
\begin{equation*} \begin{split}
d'\log(n+1)\geq \sum_{k\in [1,n]\cap E_\beta}\log(1+\alpha_k) \geq K n^{-\beta} \sharp\{ [1,n]\cap E_\beta\}
\end{split}\end{equation*}
so that
\[
\frac{1}{n}\sharp\{ [1,n]\cap E_\beta\}\leq \frac{d'}{K}\frac{\log(n+1)}{n^{1-\beta}}\to 0\quad {\text as}\quad n\to \infty\, .
\]
\end{proof}
\begin{lem}
For any $\beta\in (0,1)$ and $\gamma>1/\beta$, the set
\[
F_{\beta,\gamma}:=\big\{n\in \N\,:\,[\,n^\gamma,(n+1)^\gamma\,]\cap \N\subset E_\beta\,\big\}
\]
has vanishing density.
\end{lem}
\begin{proof}
For any $\e\in (0,1/2)$ and any integer $N\ge 2$, we have
\begin{equation*} \begin{split}
d'&\geq \frac{1}{\log(N+1)^\gamma}\sum_{(\e N)^\gamma\leq k\leq N^\gamma} \log(1+\alpha_k)\\
&\geq \frac{1}{\log(N+1)^\gamma}\sum_{n\in [\e N,N-1]\cap F_{\beta,\gamma}} \sum_{k\in [n^\gamma,(n+1)^\gamma)} \log(1+{\alpha_k})\\
&\geq \frac{K}{\gamma\log(N+1)}\sum_{n\in [\e N,N-1]\cap F_{\beta,\gamma}} \sum_{k\in [n^\gamma,(n+1)^\gamma)}k^{-\beta}\qquad N\ge 1\,.
\end{split}\end{equation*}
For any fixed $n\in [\e N,N-1]\cap F_{\beta,\gamma}$, there exists $t\in (n,n+1)$ such that $(n+1)^\gamma -n^\gamma =\gamma t^{\gamma-1}$. As $n\ge \e N$, we have $t/N>n/N\ge\e$ so that for $n\in [\e N,N-1]\cap F_{\beta,\gamma}$ the following bound holds true
\begin{equation*}\begin{split}
\sum_{n^\gamma\leq k<(n+1)^\gamma}k^{-\beta}&\geq N^{-\beta\gamma} \big((n+1)^\gamma-n^\gamma-2\big)\\ & = N^{-\beta\gamma} \big(\gamma t^{\gamma-1}-2\big)\\
&> N^{-\beta\gamma} \big(\gamma(\e N)^{\gamma-1}-2\big).
\end{split}\end{equation*}
Thus we have
\begin{equation*}\begin{split}
d'&\geq  \frac{K}{\gamma\log(N+1)}\sum_{n\in [\e N,N-1]\cap F_{\beta,\gamma}} \Big(\gamma \e^{\gamma-1}N^{(1-\beta)\gamma-1} +{\it o}(N^{(1-\beta)\gamma-1})\Big)\\
&\geq\frac{K}{\log(N+1)} \Big(\e^{\gamma-1}N^{(1-\beta)\gamma-1} +o(N^{(1-\beta)\gamma-1})\Big)
\sharp\{[\e N,{N-1}]\cap F_{\beta,\gamma}\}
\end{split}\end{equation*}
and we may deduce
\begin{equation*}\begin{split}
\frac{d'}{K\e^{1-\gamma}}
\frac{\log(N+1)}{N^{(1-\beta)\gamma}}\Big(1+o(1)\Big)\geq \frac{\sharp \{[\e N,{N-1}]\cap F_{\beta,\gamma}\}}{N}\, .
\end{split}\end{equation*}
For $N$ big enough, we have
\[
\frac{\sharp \{[\e N,{N-1}]\cap F_{\beta,\gamma}\}}{N}\leq \e
\]
and finally
\[
\frac{\sharp \{[1,{N-1}]\cap F_{\beta,\gamma}\}}{N}\leq \frac{\sharp \{[1,\e N]\cap F_{\beta,\gamma}\}}{N}+\frac{\sharp \{[\e N,{N-1}]\cap F_{\beta,\gamma}\}}{N}\leq 2\e\, .
\]
\end{proof}

\subsection{Proofs of the theorems.}
To build up a negative type function having the properties required in Theorem 1.2, we begin to construct a suitable sequence of positive definite functions.
\vskip0.2truecm\noindent
Fix $\beta\in (0,1)$ and $\gamma>1/\beta$. For any fixed $n\not\in F_{\beta,\gamma}$, we can chose $k(n)\in [\,n^\gamma,(n+1)^\gamma\,]\backslash {\bf E_\beta}$ in such a way that
\[
n^\gamma \leq k(n)\leq (n+1)^\gamma\,,\qquad k(n)^{-\beta}\geq \alpha_{k(n)}=\frac{\mu(K^{k(n)+1})-\mu(K^{k(n)})}{\mu(K^{k(n)})}\, .
\]
Denote by $\mathcal{B}(L^2(G,\mu))$ the Banach algebra of all bounded operators on the Hilbert space $L^2(G,\mu)$ and by $\lambda:G\to \mathcal{B}(L^2(G,\mu))$ the left regular unitary representation of $G$, defined as $(\lambda(s)a)(t):=a(s^{-1}t)$ for $a\in L^2(G,\mu)$ and $s,t\in G$.
\par\noindent
Consider now the function $\xi_n:=\mu(K^{k(n)})^{-1/2}\chi_{K^{k(n)}}\in L^2(G,\mu)$, with unit norm, and set
\[
\omega_n\,: G\to\R\quad \omega_n(s):=<\xi_n,\lambda(s)\xi_n>_{L^2(G)}=\frac{\mu(sK^{k(n)}\cap K^{k(n)})}{\mu(K^{k(n)})}\, .
\]
\vskip0.2truecm\noindent
By construction we have that
\par\noindent
$\bullet$ $\omega_n\in C_c(G)$ is a continuous, normalized, positive definite function with compact support. Moreover
\par\noindent
$\bullet$ $\omega_n(s)\geq 1-n^{-\beta\gamma}$ for all $s\in K$.
\vskip0.2truecm\noindent
In fact, by the translation invariance of the Haar measure and since $K^{k(n)}\supset sK^{k(n)-1}$ for $s\in K$, we have
\[
\begin{split}
\omega_n(s)&=\frac{\mu(sK^{k(n)}\cap K^{k(n)})}{\mu(K^{k(n)})} \ge \frac{\mu(sK^{k(n)}\cap sK^{k(n)-1})}{\mu(K^{k(n)})} =\frac{\mu(s(K^{k(n)}\cap K^{k(n)-1})}{\mu(K^{k(n)})} \\
&=\frac{\mu(K^{k(n)}\cap K^{k(n)-1})}{\mu(K^{k(n)})} =\frac{\mu(K^{k(n)-1})}{\mu(K^{k(n)})} =(1+\alpha_{k(n)})^{-1} \ge 1-\alpha_{k(n)}\\
&\geq 1-k(n)^{-\beta}\geq 1-n^{-\beta\gamma}\, .
\end{split}
\]
\begin{lem}
For $n\not\in F_{\beta,\gamma}$ and any integer $p\geq 1$, we have
\begin{equation}
\omega_n(s)\geq 1-pn^{-\beta\gamma}\qquad s\in K^p\, .
\end{equation}
\end{lem}
\begin{proof} Consider $s=s_1\cdots s_p\in K^p$ for some $s_1,\cdots,s_p\in K$ and notice that
\begin{equation*}\begin{split}
2\big(1-\omega(s)\big)&=||\xi_n-s\xi_n||^2=||\xi_n-s_1\cdots s_p\xi_n||^2 \\
&=||\xi_n-s_1\xi_n+s_1\xi_n-s_1s_2\xi_n+\cdots + s_1\cdots s_{p-1}\xi_n-s_1\cdots s_p\xi_n||^2\\
&\leq \big(||\xi_n-s_1\xi_n||+||\xi_n-s_2\xi_n||+\cdots +||\xi_n-s_p\xi_n||\big)^2\\
&\leq p^2\sup_{\sigma\in K} ||\xi_n-\sigma \xi_n||^2\\
&= 2p^2\sup_{\sigma\in K}\big(1-\omega_n(\sigma)\big)\leq 2p^2n^{-2\beta\gamma}.
\end{split}\end{equation*}
\end{proof}\noindent
{\it Proof of Theorem 1.1.} By previous lemma, the series $\sum_{n\not \in F_{\beta,\gamma}} \big(1-\omega_n(s)\big)$ converges uniformly on $K^p$ for any integer $p\ge 1$. Hence, it converges uniformly on each compact subset of $G$ and
\[
\ell:G\to [0,+\infty)\qquad \ell(s)=\sum_{n\not \in F_{\beta,\gamma}} \big(1-\omega_n(s)\big)
\]
is a continuous, negative definite function. The function $\ell$ is proper: in fact, if $m\in\mathbb{N}$ is greater or equal to the integer part of $(N+1)^\gamma$, then, for $s\not \in K^{2m}$, we have $\omega_n(s)=0$ for all $n\in [1,N]\backslash F_{\beta,\gamma}$ so that
\begin{equation}
\ell(s)\geq \sharp\{\, [1,N]\backslash F_{\beta,\gamma}\}\qquad s\not \in K^{2m}\, .
\end{equation}
According to Lemma 3.2 we can write $ \sharp\{\, [1,N]\backslash F_{\beta,\gamma}\}=N(1-\e_N)$ for $\e_N\to 0$. The set $\{s\,|\,\ell(s)\leq N(1-\e_N)\,\}$ is contained in $K^{2m}$, where $m$ is the integer part of $(N+1)^\gamma+1$, whose volume $\leq c\,{2^d} \big(N+1)^\gamma+1\big)^d{=O(N^{\gamma d})}$
For fixed $\e>0$ and $N$ large enough, we have, for $x\in [N-1,N]$,
\[
\mu\{\ell\leq x(1-\e)^{-1}\}\leq c'N^{\gamma d}\leq c''x^{\gamma d}
\]
so that $\mu\{\ell\leq x\} \leq c''{x^{\gamma d}}$.
\hfill $\Box$\par\noindent
{\it Proof of Theorem 1.2.} Let us chose a decreasing sequence $d_m$ converging to $d_h$, and for all $m$, chose a function $\ell_m$ such that $\ell_m(x)=O(x^{d_m})$, $x\to +\infty$ and $\sup_K \ell_d(x)=1$ (by normalization). Hence $\ell\leq k^2$ on $K^k$.
It is enough to set $\ell=\sum_m 2^{-m}\ell_m$ : $\{\ell\leq x\}\subset \{\ell_m\leq 2^m x\}$, whose measure is $O(x^{d_m})$.\par
\hfill $\Box$\par\noindent
{\it Proof of Corollary 1.3.} Recall first that the GNS space $L^2(\lambda(\Gamma)'',\tau)$ can be identified, at the Hilbert space level, with $l^2(\Gamma)$. Consider now the continuous, proper, negative type function $\ell$ constructed in Theorem 1.2. By [5 $\S$ 10.2], the quadratic form
\[
\mathcal{E}:\mathcal{F}\to [0,+\infty)\qquad \mathcal{E}[a]=\sum_{s\in\Gamma}\ell(s)\, |a(s)|^2\, ,
\]
defined on the subspace $\mathcal{F}\subset l^2(\Gamma)$ where the sum is finite, is a noncommutative Dirichlet form whose associated nonnegative, self-adjoint operator on $l^2(\Gamma)$ is the multiplication operator by $\ell$ given by
\[
{\rm dom}(L):=\{a\in l^2(\Gamma):\ell\cdot a\in L^2(G,\mu)\}\quad (La)(s):=\ell (s)\cdot a(s)\quad s\in G\, .
\]
The stated result thus follows from Theorem 1.2 and the straightforward identification ${\rm sp}(L)=\{\ell(s)\in [0,+\infty): s\in\Gamma\}$.\par
\hfill $\Box$

\newpage
\normalsize
%
%

\end{document}